\documentclass[11pt]{article}
\usepackage{amsmath, amsthm, amssymb,latexsym,mathscinet}
\usepackage{enumerate}
\usepackage{hyperref}
\usepackage{caption,subcaption,xcolor}
\usepackage{tikz,graphicx,cite,todonotes}
\usepackage{booktabs}
\usepackage{multirow}
\usepackage{setspace}

\newtheorem{thm}{Theorem}[section]
\newtheorem{lem}[thm]{Lemma}
\newtheorem{prop}[thm]{Proposition}
\newtheorem{cor}[thm]{Corollary}
\newtheorem{clm}[thm]{Claim}

\newtheorem{ques}[thm]{Question}

\numberwithin{equation}{section}

\begin{document}

\title{On induced saturation for paths}

\author{Eun-Kyung Cho\thanks{
Supported by the Basic Science Research Program through the National Research Foundation of Korea (NRF) funded by the Ministry of Education (NRF-2018R1D1A1B07043049) and the Basic Science Research Program through the National Research Foundation of Korea (NRF) funded by the Ministry of Science, ICT and Future Planning (NRF-2018R1C1B6003577).
Department of Mathematics, Hankuk University of Foreign Studies, Yongin-si, Gyeonggi-do, Republic of Korea. \texttt{ourearth3@gmail.com}
}
\and Ilkyoo Choi\thanks{Supported by the Basic Science Research Program through the National Research Foundation of Korea (NRF) funded by the Ministry of Education (NRF-2018R1D1A1B07043049), and also by the Hankuk University of Foreign Studies Research Fund.
Department of Mathematics, Hankuk University of Foreign Studies, Yongin-si, Gyeonggi-do, Republic of Korea.
\texttt{ilkyoo@hufs.ac.kr}
}
\and Boram Park\thanks{Supported by Basic Science Research Program through the National Research Foundation of Korea (NRF) funded by the Ministry of Science, ICT and Future Planning (NRF-2018R1C1B6003577).
Department of Mathematics, Ajou University, Suwon-si, Gyeonggi-do, Republic of Korea.
\texttt{borampark@ajou.ac.kr}
}
}


\date\today
\maketitle

\begin{abstract}
For a graph $H$, a graph $G$ is {\it $H$-induced-saturated} if $G$ does not contain an induced copy of $H$, but either removing an edge from $G$ or adding a non-edge to $G$ creates an induced copy of $H$.
Depending on the graph $H$, an $H$-induced-saturated graph does not necessarily exist.
In fact, Martin and Smith~\cite{MS2012} showed that $P_4$-induced-saturated graphs do not exist, where $P_k$ denotes a path on $k$ vertices.
Axenovich and Csik\'{o}s~\cite{AC2019} asked the existence of $P_k$-induced-saturated graphs for $k \ge 5$;
it is easy to construct such graphs when $k\in\{2, 3\}$.
Recently, R\"{a}ty~\cite{Raty_unpub} constructed a graph that is $P_6$-induced-saturated.

In this paper, we show that there exists a $P_{k}$-induced-saturated graph for infinitely many values of $k$.
To be precise, we find a $P_{3n}$-induced-saturated graph for every positive integer $n$.
As a consequence, for each positive integer $n$, we construct infinitely many $P_{3n}$-induced-saturated graphs.
We also show that the Kneser graph $K(n,2)$ is $P_6$-induced-saturated for every $n\ge 5$.
\end{abstract}

\section{Introduction}\label{sec:intro}

We consider only finite simple graphs.
Given a graph $G$, let $V(G)$ and $E(G)$ denote the vertex set and the edge set, respectively, of $G$.
A {\it non-edge} of $G$ is an unordered pair of vertices not in $E(G)$.
For a graph $G$ and an edge $e$ (resp. non-edge $e$) of $G$, let $G-e$ (resp. $G+e$) denote the graph obtained by removing $e$ from $G$ (resp. adding $e$ to $G$).
For a vertex $ v \in V(G)$, let $N_G(v)$ denote the set of neighbors of $v$, and let $N_G[v] = N_G(v) \cup \{v\}$.
We use $K_n$, $C_n$, and $P_n$ to denote the complete graph, a cycle, and a path, respectively, on $n$ vertices.

For a graph $H$, a graph $G$ is {\it $H$-saturated} if $G$ has no subgraph isomorphic to $H$,
but adding a non-edge to $G$ creates a subgraph isomorphic to $H$.
Since an $H$-saturated graph on $n$ vertices always exists for a particular graph $H$ when $n\ge|V(H)|$, it is natural to ask for the maximum number of edges of an $H$-saturated graph on $n$ vertices.
This gave birth to a central topic in what we now call extremal graph theory.
We hit only the highlights here.

For a graph $H$, the maximum number of edges of an $H$-saturated graph on $n$ vertices is known as the {\it Tur\'an number}  of $H$.
The Tur\'an number of an arbitrary complete graph is determined and is known as Tur\'{a}n's Theorem~\cite{Turan1941,Turan1954}, which extended the work of Mantel~\cite{Mantel1907} in 1907.
The remarkable Erd\H{o}s-Stone-Simonovits Theorem \cite{ES1946,ES1966}
resolved the Tur\'an number asymptotically for an arbitrary graph $H$ as long as $H$ is non-bipartite.
Investigating Tur\'an numbers of bipartite graphs remains to be an important, yet difficult open question; see~\cite{FS2013} for a thorough survey.  At the other extreme, the {\it saturation number} of a graph $H$ is the minimum number of edges of an $H$-saturated graph on $n$ vertices. The saturation number was introduced by Erd\H{o}s, Hajnal, and Moon in \cite{EHM1964}, where the saturation number of an arbitrary complete graph was determined.
Moreover, it is known that all saturation numbers are at most linear in $n$; see the excellent dynamic survey~\cite{FFS2011} by Faudree, Faudree, and Schmitt for more results and open questions regarding saturation numbers.

A natural generalization of the extremal questions above is to change the containment relation to induced subgraphs.
Utilizing the notion of trigraphs, Pr\"{o}mel and Steger~\cite{PS1993} and Martin and Smith~\cite{MS2012} defined the induced subgraph version of the Tur\'an number and the saturation number, respectively, of a given graph.
We omit the exact definitions here, see \cite{FFS2011,LTTZ2018,BESYY2016, EGM2019} for more details and other recent work on Tur\'an-type problems concerning induced subgraphs.

Recently, in 2019, Axenovich and Csik\'{o}s~\cite{AC2019} introduced the notion of an $H$-induced-saturated graph, whose definition emphasizes that a non-edge is as important as an edge when considering induced subgraphs.
For a graph $H$, a graph $G$ is {\it $H$-induced-saturated} if $G$ does not contain an induced copy of $H$, but either removing an edge from $G$ or adding a non-edge  to $G$ creates an induced copy of $H$.
In contrast to the fact that an $H$-saturated graph always exists for an arbitrary graph $H$, it is not always the case that an $H$-induced-saturated graph exists.
For example, a $K_n$-induced-saturated graph cannot exist when $n\geq 3$.
This is because removing an edge from a graph without $K_n$ as a (induced) subgraph cannot create a copy of $K_n$. 
It is also known that a graph $H$ with induced-saturation number zero guarantees the existence of an $H$-induced-saturated graph.
Behrens et al.~\cite{BESYY2016} studied graphs $H$ with induced-saturation number zero. 
In particular, they revealed that an $H$-induced-saturated graph exists when $H$ is a star, a matching, a 4-cycle, or an odd cycle of length at least 5.

Let us now focus our attention to paths.
We can easily see that the empty graph on at least two vertices is $P_2$-induced-saturated, and the disjoint union of complete graphs in which each component has at least three vertices
is $P_3$-induced-saturated.
Yet, a result by Martin and Smith~\cite{MS2012} implies that there is no $P_4$-induced-saturated graph. Recently, a simple proof of the aforementioned result was provided by Axenovich and Csik\'{o}s~\cite{AC2019}, where they showed various results regarding induced saturation of a certain family of trees, which unfortunately does not include paths.
In particular, they pointed out that it is unknown whether a $P_k$-induced-saturated graph exists when $k\geq 5$.

\begin{ques}[\cite{AC2019}]\label{ques:path}
Does a $P_k$-induced-saturated graph exist when $k\geq 5$?
\end{ques}

Very recently, R\"{a}ty~\cite{Raty_unpub} answered Question~\ref{ques:path} in the affirmative for $k=6$ by using an algebraic construction to generate a graph on 16 vertices that is $P_6$-induced-saturated.
The author indicated that the construction gives no idea for longer paths,
yet, we were able to generalize his construction for infinitely many longer paths.
Our main result answers Question~\ref{ques:path} in the affirmative for all $k$ that is a multiple of 3.
We now state our main theorem:

\begin{thm}\label{thm:P3n}
For every positive integer $n$, there exists a $P_{3n}$-induced-saturated graph.
\end{thm}

In fact, using the graph construction in Theorem~\ref{thm:P3n}, we can construct an infinite family of $P_{3n}$-induced-saturated graphs for each positive integer $n$.

\begin{cor}\label{cor:P3n}
For every positive integer $n$, there exist infinitely many $P_{3n}$-induced-saturated graphs.
\end{cor}

We also note that the Kneser graph $K(n,2)$ is $P_6$-induced-saturated for every $n\ge 5$.
Recall that the Kneser graph $K(n,r)$ is a graph whose vertices are the $r$-subsets of $\{1,2,\ldots,n\}$, and two vertices are adjacent if and only if the corresponding $r$-subsets are disjoint.

\begin{thm}\label{thm:P6:Kneser}
For every integer $n\ge 5$, the Kneser graph $K(n,2)$ is $P_6$-induced-saturated.
\end{thm}

In Section~\ref{sec:proof:P3n}, we present the proofs of Theorem~\ref{thm:P3n} and Corollary~\ref{cor:P3n}.
The section is divided into subsections to improve readability.
In Section~\ref{sec:proof:Kneser}, we prove Theorem~\ref{thm:P6:Kneser}.
We end the paper with some questions and further research directions in Section~\ref{sec:rmk}.
We also have the appendix where we provide proofs of some miscellaneous facts.

\section{Proofs of Theorem~\ref{thm:P3n} and Corollary~\ref{cor:P3n}}
\label{sec:proof:P3n}

In this section, we prove Theorem~\ref{thm:P3n} and Corollary~\ref{cor:P3n}.
Note that it is sufficient to prove both statements for $n\geq 2$, since the disjoint union of complete graphs in which each component has at least three vertices gives a family of infinitely many $P_3$-induced-saturated graphs.
We first prove Corollary~\ref{cor:P3n} using the graph $G_n$ in the proof of Theorem~\ref{thm:P3n}.

\begin{proof}[Proof of Corollary~\ref{cor:P3n}]
Let $G_n$ be a $P_{3n}$-induced-saturated graph in the proof of Theorem~\ref{thm:P3n}.
By Lemma~\ref{lem:ver:trans} and Proposition~\ref{prop:no_path_3n},
$G_n$ is vertex-transitive and has an induced path on $3n-1$ vertices.
For $k\ge 1$, let $H_{n,k}$ be a graph with $k$ components where each component is isomorphic to $G_n$.
Thus each component of $H_{n,k}$ is $P_{3n}$-induced-saturated, vertex-transitive, and has an induced path on $3n-1$ vertices.

Since each component is $P_{3n}$-induced-saturated and removing an edge from $H_{n,k}$ is removing an edge from a component of $H_{n,k}$,
removing an edge from $H_{n,k}$ creates an induced copy of $P_{3n}$.
By the same reason, adding a non-edge joining two vertices in a component of $H_{n,k}$ creates an induced copy of $P_{3n}$.
Adding a non-edge joining two vertices in different components of $H_{n,k}$ creates an induced copy of $P_{6n-2}$, since each component of $H_{n,k}$ is vertex-transitive and contains an induced path on $3n-1$ vertices. Hence,  $H_{n, k}$ is  $P_{3n}$-induced-saturated. \end{proof}

The proof of Theorem~\ref{thm:P3n} is split into three subsections.
We will first construct a graph $G_n$ and lay out some important properties of $G_n$ 
in Subsection~\ref{subsec:construction}.
In Subsection~\ref{subsec:proof:noP3n}, we prove that $G_n$ does not contain an induced path on $3n$ vertices.
In Subsection~\ref{subsec:pf:P3n}, we prove that the graph obtained by either adding an arbitrary non-edge to $G_n$ or removing an arbitrary edge of $G_n$ contains an induced path on $3n$ vertices.
We conclude that $G_n$ is a $P_{3n}$-induced-saturated graph.

\subsection{The construction of $G_n$ and its automorphisms}\label{subsec:construction}

Let $\mathbb{Z}_{2n}=\{0,1,\ldots,2n-1\}$, so that $\mathbb{Z}_2=\{0,1\}$.
We will always use either 0 or 1 for an element of $\mathbb{Z}_2$.
On the other hand, we will use any integer to denote an element of $\mathbb{Z}_{2n}$.
Given an element $a\in \mathbb{Z}_{2}$, let $\bar{a}$ denote the element of $\mathbb{Z}_2$ satisfying $a+\bar{a}=1$.
For simplicity, we use $(ab,j)$ to denote the element $(a,b,j)\in \mathbb{Z}_2\times  \mathbb{Z}_2\times  \mathbb{Z}_{2n}$.

For each $n\ge 2$, define a graph $G_n$ as follows.
The vertex set of $G_n$ is $\mathbb{Z}_2\times  \mathbb{Z}_2\times  \mathbb{Z}_{2n}$,
and the neighborhood of each vertex $(ab, j)\in\mathbb{Z}_2\times  \mathbb{Z}_2\times  \mathbb{Z}_{2n}$  is exactly the following set:
\[\{  (\bar{a}\bar{b},j), (a\bar{b},j), (ac,j-1), (ac,j+1), (\bar{a}c,j+2(-1)^a)  \},\]
where $c=a+b \in \mathbb{Z}_2$.
Note that we always set $a,b\in\{0,1\}$, and so the last vertex $(\bar{a}c,j+2(-1)^a)$ is either  $(\bar{a}c,j+2)$ or $(\bar{a}c,j-2)$.

We observe that $G_2$ is exactly the graph constructed by R\"{a}ty in~\cite{Raty_unpub}, see the appendix for details.
See the left figure of Figure~\ref{fig:ex} for an illustration of $G_3$.
Regarding figures representing $G_n$, we will use dots in a 2-dimensional square array to denote the vertex set of  $G_n$.
See the right figure of Figure~\ref{fig:ex} for an illustration.

\begin{figure}[h!]
  \centering
  \includegraphics[height=3cm,page=1]{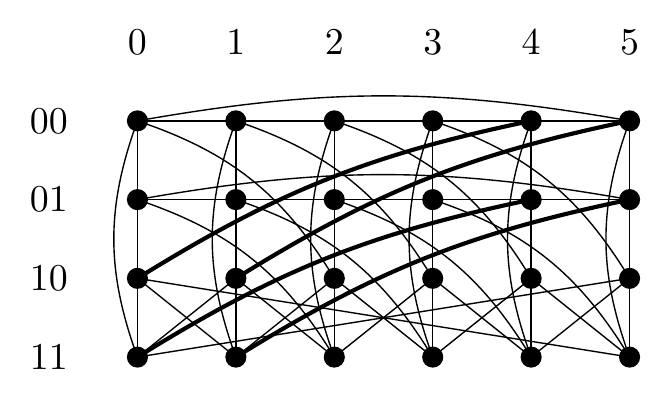}
  \quad
  \includegraphics[height=3cm,page=2]{n3+frame.pdf}
  \caption{The graph $G_3$, and a simplified representation of the vertex set of $G_n$.}
  \label{fig:ex}
\end{figure}

Now we reveal some structural properties of $G_n$. For simplicity, we will omit the parentheses of a vertex $(ab,j)$ if doing so improves the readability.
For instance, we use $N_{G_n}(ab,j)$ and $f_n(ab,j)$ instead of $N_{G_n}((ab,j))$ and $f_n((ab,j))$, respectively.

Define the following functions from $V(G_n)$ to $V(G_n)$ such that for $(ab,j)\in V(G_n)$,
\begin{eqnarray*}
f_n(ab,j)&=&(ab,j+1)\\
p_n(ab,j)&=&\begin{cases}(\bar{a}b,-j-1) &\text{if }j\equiv 0\pmod{2}\\
 (\bar{a} \bar{b},-j-1) &\text{if }j\equiv 1\pmod{2}\end{cases}\\
q_n(ab,j)&=& (ac,-j+2a), \mbox{ where } c=a+b.
\end{eqnarray*}
In the following, if there is no confusion, we denote  $f_n$, $p_n$, and $q_n$ by  $f$, $p$, and $q$, respectively. It is clear that $f$ is a bijection by definition. Since $q\circ q$ is the identity function, $q$ is also a bijection.
Moreover, $p$ is a bijection, since it has an inverse:
$$p^{-1}(ab,j)=\begin{cases}(\bar{a}\bar{b},-j-1) &\text{if }j\equiv 0\pmod{2}\\
(\bar{a} b,-j-1) &\text{if }j\equiv 1\pmod{2}.\end{cases}$$

\begin{lem}\label{lem:auto}
For $n\ge 2$, the functions  $f$, $p$, and $q$ are automorphisms of $G_n$.
\end{lem}

The proof of Lemma~\ref{lem:auto} is given in the appendix.

\begin{lem}\label{lem:ver:trans}
For $n\ge 2$,  $G_n$ is vertex-transitive; that is, for every two vertices $v$ and $v'$, there exists an automorphism $\varphi$ such that $\varphi(v)=v'$.
\end{lem}
\begin{proof}
For a positive integer $j$, let $f^j$ denote the composition $\underbrace{f\circ \cdots \circ f}_{j\text{ times}}$.
Note that $f^{j}$ maps  $(00,0)$ to $(00,j)$, $f^j\circ p\circ p$  maps  $(00,0)$ to $(01,j)$,
$f^{j+1}\circ p$ maps $(00,0)$ to $(10,j)$, and
$f^{j+1}\circ p^{-1}$ maps $(00,0)$ to $(11,j)$.
By Lemma~\ref{lem:auto}, the aforementioned four compositions are  automorphisms of $G_n$.
This completes the proof.
\end{proof}

Let us define the following sets (see Figure~\ref{fig:UL}):
\begin{eqnarray*}
U_n&=&\left( \{ (ab,j) \in V(G) \mid a=0\}\setminus\{(01,2n-1)\} \right)\setminus N_{G_n}[(00,0)]  , \\
L_n&=& V(G_n)-(U_n\cup N_{G_n}[(00,0)]).%
\end{eqnarray*}

\begin{figure}[h!]
  \centering
  \includegraphics[width=13cm]{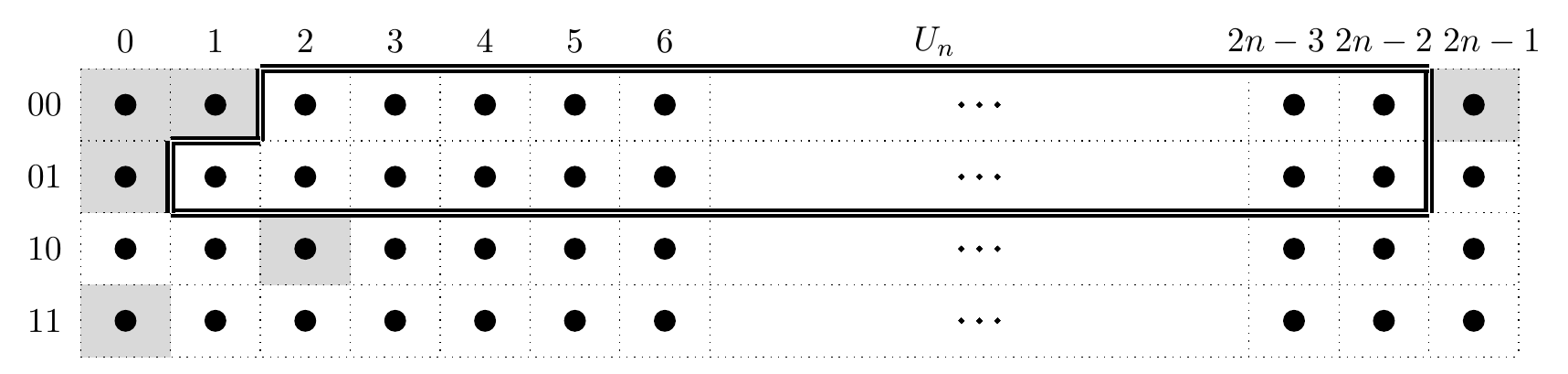}
  \caption{The vertices surrounded by the doubled lines are in $U_n$, and vertices in the other unshaded part is $L_n$. }
  \label{fig:UL}
\end{figure}

\begin{lem}\label{lem:auto:r}
For $n\ge 2$, there is an automorphism $r$ of $G_n$ satisfying that $r(00,0)=(00,0)$, $r(11,0)=(00,1)$, and  $r(U_n)\subset L_n$.
\end{lem}

The proof of Lemma~\ref{lem:auto:r} is given in the appendix.

\begin{lem}\label{lem:arc:trans}
For every $w,w'\in N_{G_n}(00,0)\setminus\{(01,0)\}$,
there is an automorphism $\varphi$ such that $\varphi(00,0)=(00,0)$ and $\varphi(w)=w'$.
\end{lem}

\begin{proof}
Note that $N_{G_n}(00,0)\setminus\{(01,0)\}$ has four vertices $w_1=(11,0)$, $w_2=(00,1)$, $w_3=(10,2)$, and $w_4=(00,2n-1)$.
Note that the automorphism $q$ of $G_n$ satisfies
$q(00,0)=(00,0)$,  $q(w_1)=w_3$, and $q(w_2)=w_4$.
By Lemma~\ref{lem:auto:r}, there is an automorphism $r$ of $G_n$ such that $r(00,0)=(00,0)$ and $r(w_1)=w_2$.
By considering compositions and inverses of $q$ and $r$, we can find an automorphism fixing $(00,0)$ and mapping $w$ to $w'$ for every $w,w'\in N_{G_n}(00,0)\setminus\{(01,0)\}$.
\end{proof}

\subsection{A longest induced path of $G_n$}\label{subsec:proof:noP3n}

In this subsection, we will show that a longest induced path of $G_n$ has exactly $3n-1$ vertices.
We  first provide a set of vertices that induces a path on $3n-1$ vertices.
As one can see in Figure~\ref{fig:long:path}, the set
$\displaystyle\left(\bigcup_{i=0}^ {n-1}  \{ (00,2i), (01,2i)\} \right)\cup
\left(\bigcup_{i=0}^{\lceil\frac{n}{2}\rceil-2}\{(01,4i+1), (00,4i+3)\}\right)\cup S$
 induces a path on $3n-1$ vertices, where
\[ S=
\begin{cases}
\{(01,2n-3)\} &\text{if $n$ is even}\\
\emptyset&\text{if $n$ is odd.}\\
\end{cases}
\]

\begin{figure}[h!]
  \centering
  \includegraphics[page=1,width=8.4cm]{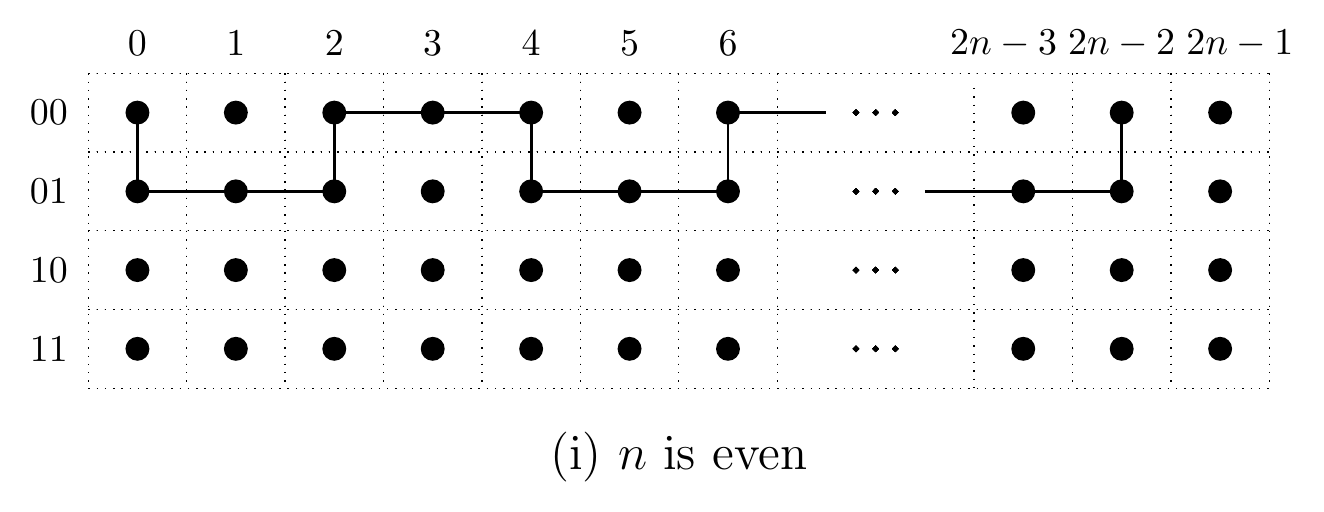}
  \includegraphics[page=2,width=8.4cm]{long_paths.pdf}
\caption{An induced path on $3n-1$ vertices in $G_n$.
  }
  \label{fig:long:path}
\end{figure}

We now show that $G_n$ does not have an induced path on $3n$ vertices.

\begin{prop}\label{prop:no_path_3n}
For $n\ge 2$, a longest path of $G_{n}$ has exactly $3n-1$ vertices.
\end{prop}
\begin{proof}
It is sufficient to show that $G_n$ does not have an induced path on $3n$ vertices.
Suppose to the contrary that for some $n$, $G_n$ has an induced path $P$ on $3n$ vertices, and let $P:v_0v_1\ldots v_{3n-1}$. We take the minimum such $n$.
Since $G_2$ is isomorphic to the graph in~\cite{Raty_unpub} which is $P_6$-induced-saturated (see also the appendix), we have $n\geq 3$.
\begin{clm}\label{claim:new}
If there is an induced path $Q$ of $G_n$ in the unshaded part of the left figure of Figure~\ref{fig:new:claim},
then $Q$ has at most $3n-4$ vertices.
Moreover,
if such $Q$ starts with one vertex in  $\{a_i,b_i\}$
and $|Q\cap \{a_i,b_i\}| = 1$ for some $i\in\{1,2\}$,
then $Q$ has at most $3n-5$ vertices.
\end{clm}
\begin{figure}[h!]
  \centering
  \includegraphics[height=2.8cm, page=1]{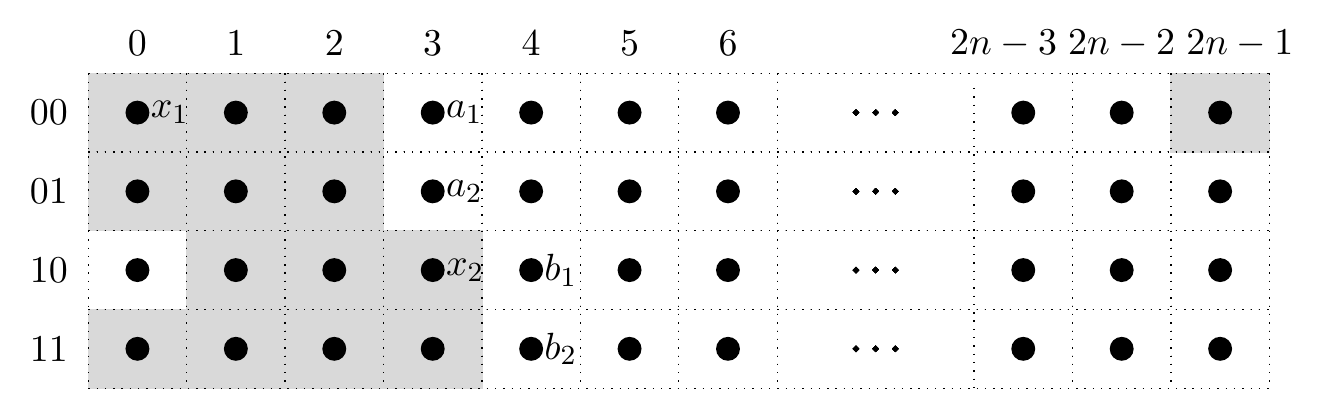}
  \includegraphics[height=2.8cm, page=2]{fig_new_claim.pdf}
  \caption{Illustrations for Claim~\ref{claim:new}.}
  \label{fig:new:claim}
\end{figure}
\begin{proof}
We use the labels of the vertices in Figure~\ref{fig:new:claim}.
Let $Q$ be an induced path of $G_n$ in the unshaded part of the left figure of Figure~\ref{fig:new:claim}.
We consider the graph $H$ that is isomorphic to $G_{n-1}$, and whose vertices are labeled as the right figure of Figure~\ref{fig:new:claim}.
Let $Q'$ be the subgraph of $H$, which is the embedding of $Q$ into $H$. Note that $Q'$ is an induced path of $H$, isomorphic to $Q$.
By the minimality of $n$, $Q'$ has at most $3n-4$ vertices, and so does $Q$.

Suppose that $Q$ starts with one vertex in $\{a_i,b_i\}$ and $|Q \cap \{a_i,b_i\}| = 1$ for some $i \in \{1, 2\}$.
Then $Q'$ starts with one vertex in $\{a_i',b_i'\}$ and $|Q'\cap \{a_i',b_i'\}| =1$ for some $i\in\{1,2\}$.
Then $x_i'+Q'$ is an induced path of $H$, since the neighbors of $x_i'$ in the unshaded part are $a_i'$ and $b_i'$.
Hence, if $|V(Q')|\ge 3n-4$, then $x_i'+Q$ is an induced path on at least $3n-3$ vertices, which is a contradiction to the minimality of $n$.
\end{proof}

By Lemma~\ref{lem:ver:trans}, we may assume  $v_0=(00,0)$.
Furthermore, by Lemma~\ref{lem:arc:trans}, we may assume either $v_1 = (01,0)$ or $v_1=(00,1)$.

Suppose $v_1=(01,0)$.
Let $S=N_{G_{n}}(v_0)\cup N_{G_{n}}(v_1)$, which are the vertices marked with an $\times$ in Figure~\ref{fig:v:010}.
The graph $G_{n}-S$ is the disjoint union of $K_2$ and a graph $H$, where $H$ is induced by the unshaded part of Figure~\ref{fig:v:010}.
We know $S$ contains at most three vertices of $P$,  so $H$ must contain an induced path on $3n-3$ vertices.
Yet, by the minimality of $n$,  $G_{n-1}$ does not contain an induced path on $3n-3$ vertices. Since $H$ is an induced subgraph of $G_{n-1}$, $H$ also does not contain an induced path on $3n-3$ vertices, which is a contradiction.
Thus $v_1=(00,1)$, and $v_2 \in \{(01,1), (00,2),(10,3), (11,1)\}$.
\begin{figure}[h!]
  \centering
  \includegraphics[width=11cm,page=1]{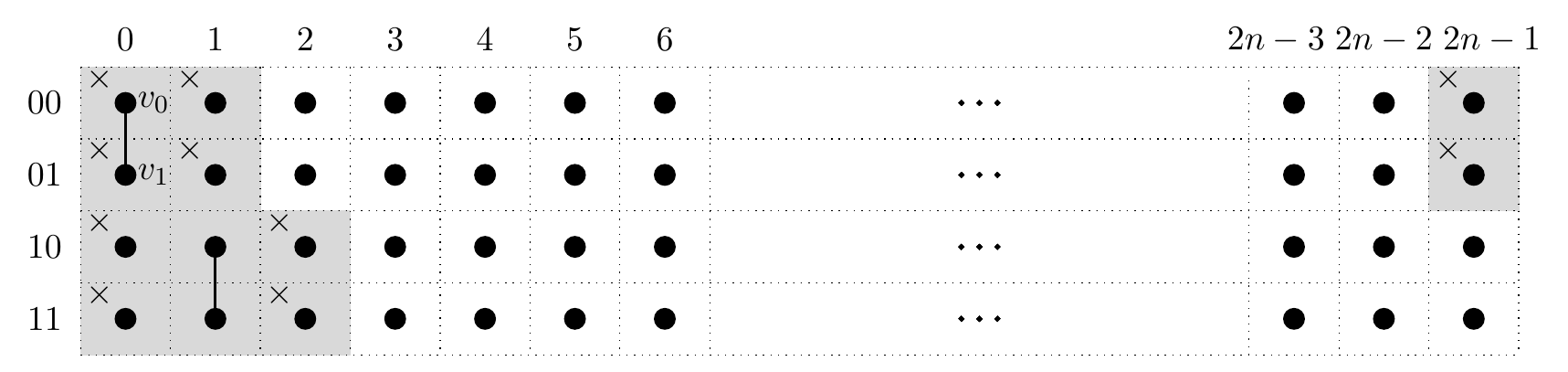}
  \caption{The unshaded part is isomorphic to an induced subgraph of $G_{n-1}$ }
  \label{fig:v:010}
\end{figure}

See the left figure of Figure~\ref{fig:v:011} and we use the labels of the vertices as in the figure.
The vertices marked with an $\times$ do not belong to $V(P)$ as they are in $N_{G_n}(v_0) \setminus \{v_1\}$.
Note that exactly one of $y_i$, that is $v_2$, belongs to $V(P)$.
\begin{figure}[h!]
  \centering
  \includegraphics[height=3.15cm,page=2]{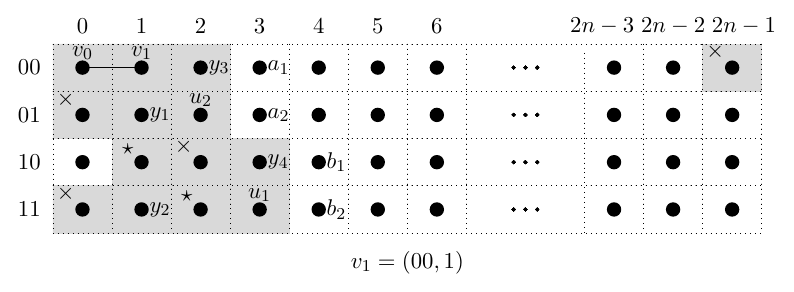}
  \includegraphics[height=3.15cm,page=3]{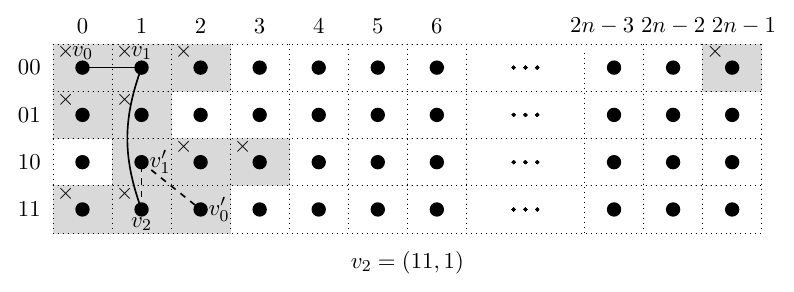}
  \caption{An illustration when $v_1=(00,1)$. When $v_2=(11,1)$, the shaded part does not contain a vertex in  $V(P)\setminus\{v_0,v_1,v_2\}$.}
\label{fig:v:011}
\end{figure}

If $v_3\in \{a_1,a_2,b_1,b_2\}$, then $v_2\in \{y_3,y_4\}$ and so
$V(P)\setminus\{v_0,v_1,v_2\}$ are in the unshaded part of the left figure of  Figure~\ref{fig:v:011}, which is a contradiction to Claim~\ref{claim:new}.
Hence, $v_3\not\in \{a_1,a_2,b_1,b_2\}$.

Let $Q:v_4v_5\ldots v_{3n-1}$, which is is an induced path of $G_n$ on $3n-4$ vertices.
Suppose that $v_2\in \{y_1,y_3,y_4\}$.
Since $v_3\not\in \{a_1,a_2,b_1,b_2\}$, $v_3=u_i$ for some $i\in \{1,2\}$.
It is clear that $v_4\in \{a_i,b_i\}$, so $|V(Q)\cap \{a_i,b_i\}| = 1$.
Note that $V(Q)$ is contained in the unshaded part of the left figure of Figure~\ref{fig:v:011}.
This is a contradiction to the `moreover' part of Claim~\ref{claim:new}. Hence, $v_2=y_2=(11,1)$.

Now see the right figure of Figure~\ref{fig:v:011}.
Note that the vertices marked with an $\times$ are in $N_{G_n}(v_0)\cup N_{G_n}(v_1)$ and so the shaded part of the figure  does not contain a vertex in  $V(P)\setminus\{v_0,v_1,v_2\}$.
We consider a path $P'$ obtained from $P$ by replacing $v_0$ and $v_1$ with $v'_0=(11,2)$ and $v'_1=(10,1)$, respectively. Then $P'$ is also an induced path of $G_{n}$ on $3n$ vertices, since the neighbors of $v_0'$ and $v'_1$ are in the shaded part.
Then the image of $P'$ under the automorphism $f^3\circ p^{-1}$ is an induced path starting with $v_0,v_1,y_1$, since
\begin{eqnarray*}
&&f^3\circ p^{-1}(v'_0)=  f^3\circ p^{-1} (11,2)
= f^3(00,-3)=(00,0)=v_0\\
&&f^3\circ p^{-1}(v'_1)=  f^3\circ p^{-1} (10,1)
= f^3(00,-2)=(00,1)=v_1\\
&&f^3\circ p^{-1}(v_2)= f^3\circ p^{-1} (11,1) =
f^3 (01,-2)=(01,1)=y_1.
\end{eqnarray*}
We can reach a contradiction by the same argument of the previous paragraph.
\end{proof}

\subsection{A longest induced path of the graph obtained from $G_n$ by either adding a non-edge or removing an edge}\label{subsec:pf:P3n}

\begin{prop}\label{prop:G-e:3n}
For $n\ge 2$ and every edge $e$ of $G_n$, $G_n-e$ contains an induced path on $3n$ vertices.
\end{prop}

\begin{proof}
Take an edge $e=vw$ of $G_n$.
By Lemmas~\ref{lem:ver:trans}~and~\ref{lem:arc:trans}, we may assume that   $v=(00,0)$ and $w\in \{ (01,0), (00,2n-1)\}$.

Suppose that $w=(01,0)$.
See Figure~\ref{fig:long:path:G-e}~(i).
If $n$ is even,
then by adding the vertex $(00,2n-1)$
(the circled vertex of the left figure of Figure~\ref{fig:long:path:G-e}~(i)) to the path in Figure~\ref{fig:long:path}~(i),
we obtain an induced path of $G-e$ on $3n$ vertices.
If $n$ is odd,
then the path obtained from the path in Figure~\ref{fig:long:path}~(ii) by adding the vertices $(00,2n-1)$
and $(11,0)$  (the circled vertices of the right figure of Figure~\ref{fig:long:path:G-e}~(i))
and by deleting the vertex $(01,2n-2)$ (the shaded vertex of the right figure of Figure~\ref{fig:long:path:G-e}~(i)),
we obtain an induced path of $G-e$ on $3n$ vertices.
Suppose that the endpoints of $w=(00,2n-1)$.
See Figure~\ref{fig:long:path:G-e}~(ii).
If $n$ is even,
then by adding the vertex $(00,2n-1)$ (the circled vertex of the left figure of Figure~\ref{fig:long:path:G-e}~(ii)) to the path in Figure~\ref{fig:long:path}~(i), we obtain an induced path of $G-e$ on $3n$ vertices.
If $n$ is odd,
then the path obtained from the path in Figure~\ref{fig:long:path}~(ii)
by adding the vertices $(00,2n-1)$ and $(11,2n-1)$  (the circled vertices of the right figure of Figure~\ref{fig:long:path:G-e}~(ii))
and by deleting the vertex $(01,2n-2)$ (the shaded vertex of the right figure of Figure~\ref{fig:long:path:G-e}~(ii)) is an induced path of $G-e$ on $3n$ vertices.
\end{proof}
\begin{figure}[h!]
 (i) $w=(01,0)$:\\[1ex]
  \includegraphics[height=3.2cm,page=1]{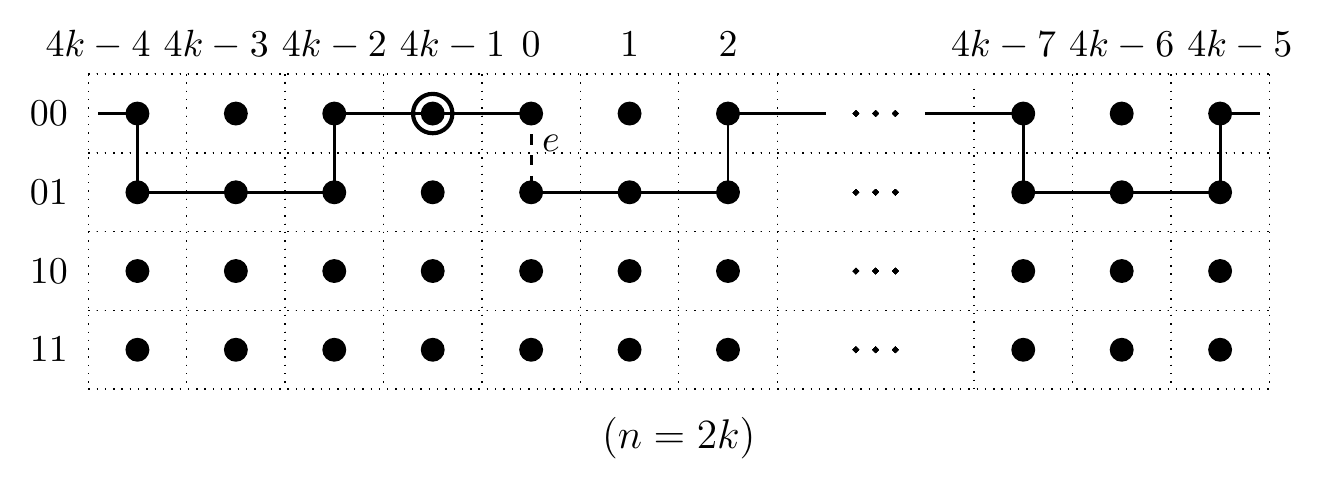}
    \includegraphics[height=3.2cm,page=2]{fig_del_edge.pdf}
  \\[2ex]
   (ii) $w=(00,1)$:\\[1ex]
    \includegraphics[height=3.2cm,page=3]{fig_del_edge.pdf}
  \includegraphics[height=3.2cm,page=4]{fig_del_edge.pdf}
  \caption{An induced path on $3n$ vertices in $G_n-e$, where $e=vw$ is the dashed edge.}
  \label{fig:long:path:G-e}
\end{figure}

\begin{prop}\label{prop:G+e:3n}
For $n\ge 2$ and every non-edge $e$ of $G_n$, $G_n+e$ contains an induced path on $3n$ vertices.
\end{prop}

\begin{proof}
Consider a non-edge $e$ of $G_n$, which we may assume one endpoint is $(00,0)$ by Lemma~\ref{lem:ver:trans}.
Note that $U_n\cup L_n$ is the set of all non-neighbors of $(00,0)$, where $U_n$ and $L_n$ are the sets defined before Lemma~\ref{lem:auto:r}.
By Lemma~\ref{lem:auto:r}, it is sufficient to show that $G_n+e$ contains an induced path on $3n$ vertices, where the other endpoint $w$ of $e$ is in $L_n$.
In the following, we consider five cases according to $w$.  For each case, we provide two figures, depending on the parity of $n$.
Each figure shows the edges of an induced path $P$ of $G_n$ on $3n-1$ vertices such that $(00,0)$ is an end vertex of $P$.
Now, adding a non-edge joining $(00,0)$ and a circled vertex $w$ to $G_n$ extends $P$ to an induced path on $3n$ vertices, which proves the proposition.  We remark that the shaded part of each figure represents a repeated pattern of the path.
\begin{itemize}
    \item[(i)] $w$ is either $(10,4t)$ or $(11,4t+2)$ for some $t$
\end{itemize}
  \includegraphics[height=2.6cm,page=1]{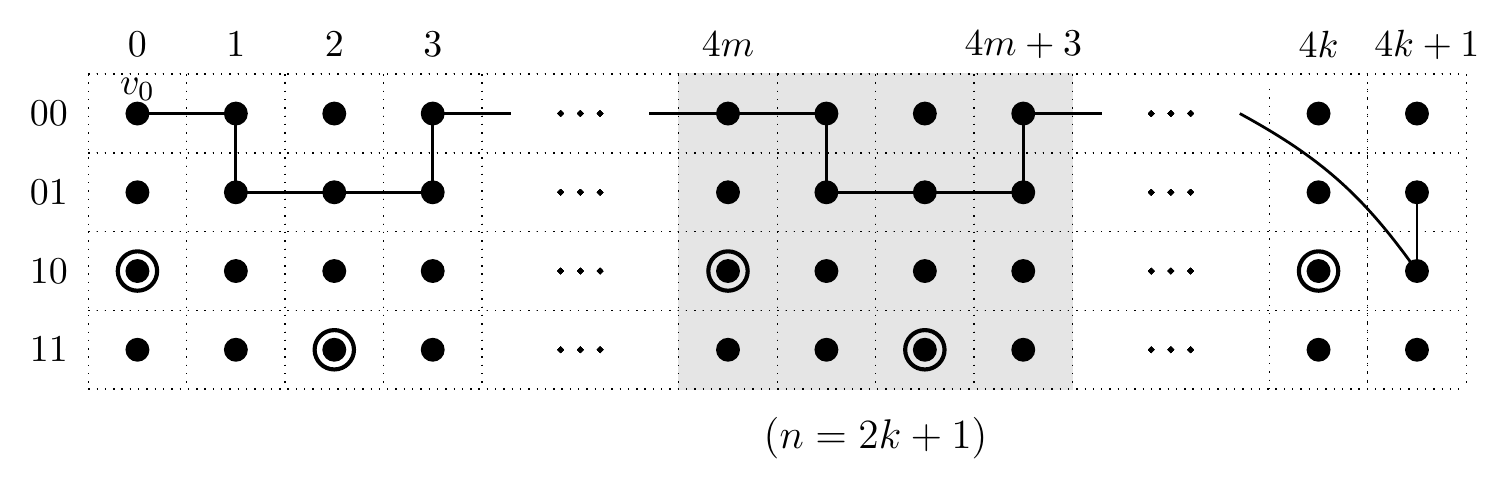}   \includegraphics[height=2.6cm,page=2]{fig_add_edge.pdf}
\begin{itemize}
    \item[(ii)] $w$ is either $(10,4t+1)$ or $(11,4t+3)$ for some $t$
\end{itemize}
    \includegraphics[height=2.6cm,page=3]{fig_add_edge.pdf}
    \includegraphics[height=2.6cm,page=4]{fig_add_edge.pdf}
\begin{itemize}
\item[(iii)] $w$ is either  $(10,4t+2)$ or $(11,4t)$ for some $t$
\end{itemize}
  \includegraphics[height=2.6cm,page=5]{fig_add_edge.pdf}
  \includegraphics[height=2.6cm,page=6]{fig_add_edge.pdf}
\begin{itemize}
\item[(iv)] $w$ is either  $(10,4t+3)$ or $(11,4t+1)$ for some $t$
\end{itemize}
   \includegraphics[height=2.6cm,page=7]{fig_add_edge.pdf}
  \includegraphics[height=2.6cm,page=8]{fig_add_edge.pdf}
\begin{itemize}
\item[(v)]  $w=(01,2n-1)$
\end{itemize}
   \includegraphics[height=2.6cm,page=9]{fig_add_edge.pdf}
   \includegraphics[height=2.6cm,page=10]{fig_add_edge.pdf}
\end{proof}

Hence, from Propositions~\ref{prop:no_path_3n},~\ref{prop:G-e:3n},~and~\ref{prop:G+e:3n}, Theorem~\ref{thm:P3n} follows.

\section{Proof of Theorem~\ref{thm:P6:Kneser}}\label{sec:proof:Kneser}

Recall that $K(n, 2)$ denotes the Kneser graph and let $n\ge 5$ be an integer.
First, we will show that $K(n,2)$ has no induced path on 6 vertices.
Suppose to the contrary that $K(n, 2)$ has an induced path $P: v_1v_2\ldots v_6$.
Without loss of generality, we may assume $v_1=\{1,2\}, v_2=\{3,4\}$, and $v_3=\{1,5\}$.
Since $v_4\cap v_3=\emptyset$ and $v_4\cap v_i\neq \emptyset$ for every $i\in \{1, 2\}$,
we have  $2\in v_4$, and so without loss of generality we may assume $v_4=\{2,3\}$.
Similarly, since $v_5\cap v_4=\emptyset$ and $v_5\cap v_i\neq\emptyset$ for every $i\in\{1,2\}$, we know $v_5=\{1,4\}$.
Since $v_6\cap v_5=\emptyset$ and $v_6\cap v_i\neq\emptyset$
 for every $i\in\{1,2\}$,
we know $v_6=\{2,3\}=v_4$, which is a contradiction.

To show that deleting an arbitrary edge creates an induced $P_6$, we may assume the edge joining  $\{2,3\}$ and $\{4,5\}$ is deleted, since $K(n,2)$ is arc-transitive.
Now, $\{4,5\}$, $\{1,2\}$, $\{3,4\}$, $\{1,5\}$, $\{2,3\}$, $\{1,4\}$ form an induced $P_6$.
To show that adding a non-edge creates an induced $P_6$, we may assume that the non-edge joining
$\{1,2\}$ and $\{1,3\}$ is added, since $K(n, 2)$ is arc-transitive.
Then $\{1,3\}$, $\{1,2\}$, $\{3,4\}$, $\{1,5\}$, $\{2,3\}$, $\{1,4\}$ form an induced $P_6$. Hence, the Kneser graph $K(n,2)$ is  $P_6$-induced-saturated for all $n\ge 5$.

\section{Remarks and further research directions}\label{sec:rmk}

In this paper, we found a $P_k$-induced-saturated graph for infinitely many values of $k$.
Yet, there are still many values of $k$ for which we do not know if a $P_k$-induced-saturated graph exists.
The ultimate goal is to answer the following question.

\begin{ques}
Can we classify all positive integers $k$ for which a $P_k$-induced-saturated graph exists?
\end{ques}

As we know there exists a $P_k$-induced-saturated graph when $k\in\{2,3\}$ and there is no $P_k$-induced-saturated graph when $k=4$, the first open case of the above question is when $k=5$.

\begin{ques}
Does there exist a $P_5$-induced-saturated graph?
\end{ques}

The graph $G_n$ (in Subsection~\ref{subsec:construction}) is $5$-regular.
In other words, there are infinitely many integers $k$ such that there is a 5-regular $P_k$-induced-saturated graph.

The Kneser graph $K(5,2)$ in Theorem~\ref{thm:P6:Kneser}, which is $P_6$-induced-saturated, is 3-regular.
For many values of $k$, we found a 3-regular $P_k$-induced-saturated graph with the aid of a computer.
For example,
the Heawood graph is $P_k$-induced-saturated for $k\in\{8,9\}$,
the generalized Petersen graph $GP(10,3)$ is $P_k$-induced-saturated for $k\in\{12,13\}$,
the Coxeter graph is $P_{19}$-induced-saturated,
the Dyck graph is $P_{k}$-induced-saturated for $k\in\{21,22\}$, and
the generalized Pertersen graphs $GP(17,3)$ and $GP(17,6)$ are $P_{23}$-induced-saturated.
These findings made us wonder if the following question is true.

\begin{ques}
Are there infinitely many integers $k$ for which there is a 3-regular $P_k$-induced-saturated graph?
\end{ques}

Let $Q^d_{n}$ be the graph obtained by adding the diagonals to the $n$-dimensional hypercube $Q_n$.
The graph constructed by R\"{a}ty in~\cite{Raty_unpub} is isomorphic to $Q^d_{4}$.
A computer check confirms that $Q^d_5$ is a $P_k$-induced-saturated graph for $k\in \{12,13,14\}$.
We think $Q^d_n$ is a good candidate to investigate, and put forward a question to encourage research in this direction.

\begin{ques}
For every $n \geq 4$, does there exist a $k$ such that $Q^d_n$ is a $P_k$-induced-saturated graph?
\end{ques}

\section*{Appendix}
\subsection*{The $P_6$-induced-saturated graph given by R\"{a}ty in \cite{Raty_unpub}}

Let $\mathbb{F}_{16}=\mathbb{F}_2(\alpha)/(\alpha^4+\alpha+1)$ be the finite field with 16 elements.
Then the multiplicative group $\mathbb{F}_{16}^{\times}$ is generated by $\alpha$, and  $(\mathbb{F}_{16}^{\times})^3=\{1,\alpha^3, \alpha^3+\alpha,\alpha^3+\alpha^2, \alpha^3+\alpha^2+\alpha+1\}.$
In \cite{Raty_unpub}, R\"{a}ty provided a $P_6$-induced-saturated graph $H$,  defined by
$V(H)=\mathbb{F}_{16}$ and $E(H)=\{ab\mid a+b\in (\mathbb{F}_{16}^{\times})^3 \}$.
One can check that  $H$ is isomorphic to  $G_2$ by Figure~\ref{fig:G2}, which shows the image of an isomorphism from $G_2$ to $H$.

\begin{figure}[h!]
\centering
    \includegraphics[height=3cm]{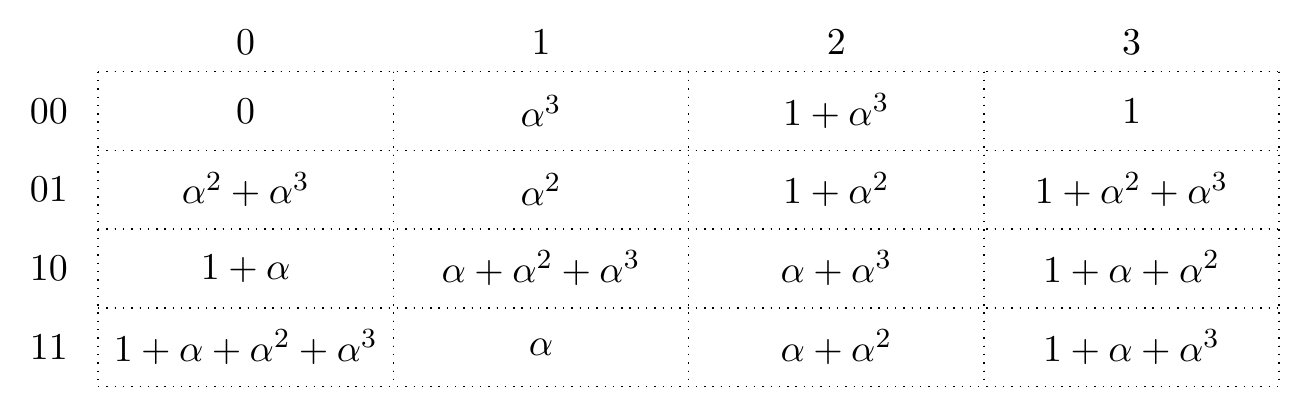}
  \caption{The image of an isomorphism from $G_2$ to $H$.}
  \label{fig:G2}
\end{figure}

\subsection*{Proof of Lemma~\ref{lem:auto}}

It is sufficient to check that each of $f$, $p$, and $q$ moves an edge to an edge.
Fix a vertex $(ab,j)\in V(G_n)$ and let $c=a+b\in\mathbb{Z}_2$.
See Tables~\ref{table:f(N(v))}~and~\ref{table:N(f(v))}.

\begin{table}[h!]\centering
\footnotesize{\begin{tabular}{p{18mm}||p{19mm}|p{19mm}|p{22mm}|p{25mm}|p{30mm}}
\qquad $v$ & $(\bar{a}\bar{b},j)$ & $(a\bar{b},j)$ & $(ac,j-1)$ & $(ac,j+1)$ & $(\bar{a}b,j+2(-1)^a)$ \\   \hline   \hline
  $f(v)$ &  $(\bar{a}\bar{b},j+1)$ &  $(a\bar{b},j+1)$  & $(ac,j)$   & $(ac,j+2)$ &   $(\bar{a}b,j+2(-1)^a+1)$ \\   \hline
  $p(v)$, $j$: even &
  $(a\bar{b},-j-1)$ & $(\bar{a}\bar{b},-j-1)$ & $(\bar{a}\bar{c},-j)$ & $(\bar{a}\bar{c},-j-2)$ & $(ab,-j-2(-1)^a-1)$  \\    \hline
  $p(v)$, $j$: odd & $(ab,-j-1)$ & $(\bar{a}b,-j-1)$ & $(\bar{a}c,-j)$ & $(\bar{a}c,-j-2)$ & $(a\bar{b},-j-2(-1)^a-1)$ \\    \hline
  $q(v)$ & $(\bar{a}c,-j+2\bar{a})$  & $(a\bar{c},-j+2a)$ & $(ab,-j+1+2a)$ & $(ab,-j-1+2a)$ & $(\bar{a}\bar{c},-j-2(-1)^a+2\bar{a})$  \\
\end{tabular}}\caption{The values $f(v)$, $p(v)$, and $q(v)$ for $v\in N_{G_n} (ab,j)$.}\label{table:f(N(v))}
\end{table}

\begin{table}[h!]
\footnotesize{\begin{tabular}{l||l}
$N_{G_n}(f(ab,j))$ &   \multirow{2}{*}{$(\bar{a}\bar{b},j+1), (a\bar{b},j+1),
(ac,j), (ac,j+2), (\bar{a}b,j+1+2(-1)^a)$} \\
\quad  $(=N_{G_n} ( ab,j+1))$ &   \\ \hline
$N_{G_n}(p(ab,j))$, $j:$ even    &
\multirow{2}{*}{$(a\bar{b},-j-1), ( \bar{a}\bar{b},-j-1), ( \bar{a}\bar{c},-j-2), ( \bar{a}\bar{c},-j), (ab,-j-1+2(-1)^{\bar{a}}) $ }
\\
\quad   $(=N_{G_n} ( \bar{a}b,-j-1))$& \\ \hline
$N_{G_n}(p(ab,j))$, $j:$ odd & \multirow{2}{*}{$(ab,-j-1), ( \bar{a}b,-j-1),  (\bar{a}c,-j-2), ( \bar{a}c,-j), ( a\bar{b},-j-1+2(-1)^{\bar{a}}) $}\\
\quad  $(=N_{G_n} ( \bar{a}\bar{b},-j-1))$ & \\ \hline
$N_{G_n}(q(ab,j))$  &
\multirow{2}{*}{$(\bar{a}\bar{c},-j+2a), ( a\bar{c},-j+2a), (ab,-j+2a-1),(ab,-j+2a+1),
( \bar{a}c,-j+2a+2(-1)^a)$} \\
\quad   $(=N_{G_n} ( ac,-j+2a) )$& \\
\end{tabular}}\caption{$N_{G_n}(f(ab,j))$, $N_{G_n}(p(ab,j))$, and $N_{G_n}(q(ab,j))$.}\label{table:N(f(v))}
\end{table}

One can check that the set of the five vertices in the $i$th row of Table~\ref{table:f(N(v))}  is equal to that of the vertices in the $i$th row of Table~\ref{table:N(f(v))}.
Especially when one compares the last rows of the two tables, note that
$-j-2(-1)^a+2\bar{a}=-j+2a$ and $-j+2\bar{a}=-j+2a+2(-1)^a$, which implies that the last vertex of one table matches the first vertex of the other table.

\subsection*{Proof of Lemma~\ref{lem:auto:r}}

For simplicity, let $A=T_0\cup T_1$, where
$T_j=\{(00,j),(01,j),(10,j+1),(11,j+1)\}$ for  $j\in \{0,1\}$.
We define $r:V(G_n)\rightarrow V(G_n)$ by
\[ r(ab,j)=\begin{cases}
q(ab,j)&\text{if }(ab,j)\in A\\
f^2\circ p^{-1}(ab,j)&\text{otherwise.}
\end{cases}\]
It is clear that $r$ is well-defined, and
Figure~\ref{fig:r} shows the image of $r$.
To be precise,
\[ r(ab,j)=\begin{cases}
(ac,-j+2a)&\text{if }(ab,j)\in A,\\
(\bar{a}\bar{b},-j+1)&\text{if  }(ab,j)\not\in A \text{ and } j\text{ is even},\\
(\bar{a}b,-j+1)&\text{if  }(ab,j)\not\in A\text{ and } j\text{ is odd}.\\
\end{cases}\]
We will show that $r$ is an automorphism of $G_n$, which proves Lemma~\ref{lem:auto:r}.

\begin{figure}[h!]
  \centering
  \includegraphics[width=16cm]{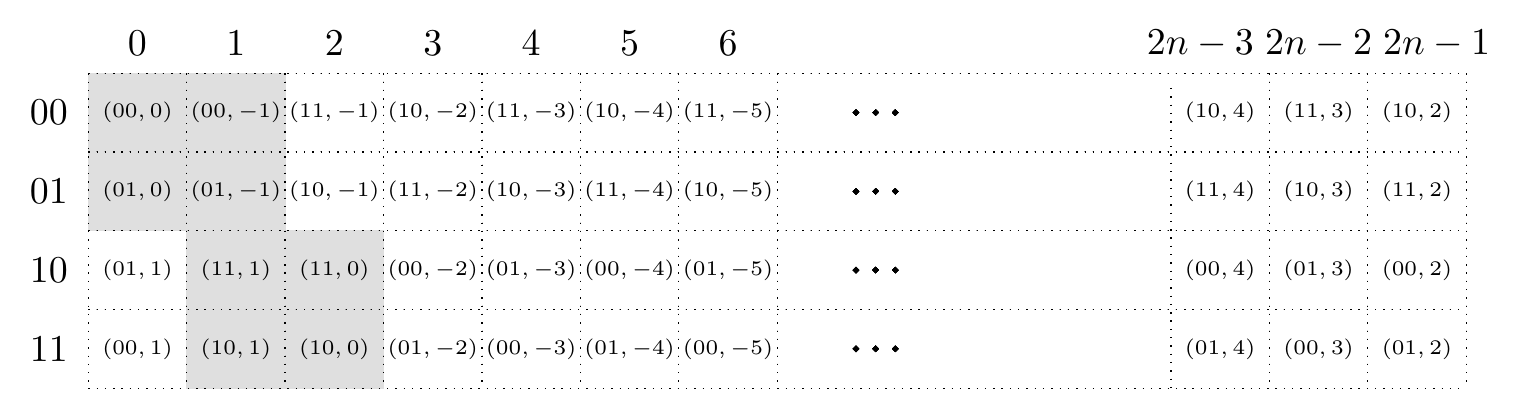}
  \caption{The image of the mapping $r$, where the shaded part shows the image of $A$ under $r$.}
  \label{fig:r}
\end{figure}

We first show that $r$ is a bijection.
It is sufficient to check that if $(ab,j)\neq (a'b',j')$ then $r(ab,j)\neq r(a'b',j')$.
Since we know that $q$ and $f^2\circ p^{-1}$ are bijections, it is enough to check the case where $(a'b',j')\in A$ and $(ab,j)\not\in A$.
Suppose that $r(a'b',j)=r(ab,j)$.
Then, $q(a'b',j)=f^2\circ p^{-1}(ab,j)$, which implies that
$a'=\bar{a}$ and $-j'+2a'=-j+1$.
Suppose that $a=0$. Then $a'=1$ and $j'=j+1$.
Since $(a'b',j')\in A$ and $a'=1$,
$j'\in\{1,2\}$.
Then we have $j\in \{0,1\}$, which is a contradiction to $(ab,j)\not\in A$.
Similarly, for the case where $a=1$, we have $a'=0$ and $j'=j-1$, which also implies that
$j\in \{1,2\}$, a contradiction to $(ab,j)\not\in A$.
Hence, $r$ is an injection.

Now we check that $r$ preserves an edge.
Since $q$ and $f^2\circ p^{-1}$ are automorphisms of $G_n$, it is sufficient to check that $r(v)r(w)$ is an edge for every edge $vw$ of $G_n$, where $v\in A$ and $w\in N_{G_n}(v)\setminus A$. For each vertex $v\in A$, there are two such neighbors $w$ and one can check all cases from Table~\ref{table:r}.
\begin{table}[h!]\centering
\footnotesize{\begin{tabular}{c|l||c|l}
$v\in A$ & \quad $w\in N_{G_n}(v)\setminus A$ \quad & $r(v)=q(v)$ & \quad $r(w)=f^2\circ p^{-1}(w)$ \quad \\ \hline   \hline
$(00,0)$&\quad   $(00,-1)$, $(11,0)$ & $(00,0)$ &\quad  $(10,2)$, $(00,1)$ \\
$(01,0)$ &\quad  $(01,-1)$, $(10,0)$ &
$(01,0)$ &\quad  $(11,2)$, $(01,1)$\\
$(00,1)$ &\quad  $(00,2)$, $(10,3)$ &
$(00,-1)$&\quad  $(11,-1)$,  $(00,-2)$\\
$(01,1)$&\quad $(01,2)$, $(11,3)$    &  $(01,-1)$ &\quad $(10,-1)$, $(01,-2)$    \\
$(10,1)$&\quad $(11,0)$, $(00,-1)$    & $(11,1)$  &\quad $(00,1)$, $(10,2)$    \\
$(11,1)$&\quad $(10,0)$, $(01,-1)$     &  $(10,1)$  &\quad  $(01,1)$, $(11,2)$  \\
$(10,2)$&\quad $(01,2)$, $(11,3)$    & $(11,0)$  &\quad  $(10,-1)$, $(01,-2)$  \\
$(11,2)$&\quad  $(00,2)$, $(10,3)$   & $(10,0)$  &\quad  $(11,-1)$, $(00,-2)$  \\
\end{tabular}}\caption{The values $r(v)$ and $r(w)$  for $v\in A$ and $w\in N_{G_n}(v)\setminus A$.}\label{table:r}
\end{table}

It remains to show that $r(U_n)\subset L_n$, that is, $r(v)\in L_n$ for all $v\in U_n$.
Since $U_n \cap A$ contains a unique vertex $(01,1)$ and $r(01,1)=(01,2n-1)$
as in Figure~\ref{fig:r},  we know that
$r(v)\in L_n$ for $v\in U_n \cap A$.
For a vertex $v\in U_n \setminus A$,
say $v=(0b,j)$ for some $b\in\mathbb{Z}_2$ and $j\in \mathbb{Z}_{2n}$,
by the definition of $r$ we have $r(0b,j)=(1b',-j+1)$ for some $b'\in \mathbb{Z}_2$, and so $r(0b,j)\in L_n$.
This completes the proof.

\end{document}